\documentclass[12pt]{amsart}
\usepackage[utf8]{inputenc}
\usepackage{amsfonts}
\usepackage{amsmath}
\usepackage{amssymb}
\usepackage{amsthm}
\usepackage{xcolor}
\usepackage[margin=1.3in]{geometry}
\usepackage{graphicx}
\usepackage[11pt]{moresize}
\usepackage{tikz}
\tikzstyle{vertex}=[circle,draw=black,fill=black,inner sep=0,minimum size=3pt,text=white,font=\footnotesize]
\usepackage{nccmath}

\newtheorem{theorem}{Theorem}
\newtheorem*{conjecture*}{Conjecture}
\newtheorem{proposition}{Proposition}[section]
\newtheorem{lemma}[proposition]{Lemma}

\newtheorem{remark}[proposition]{Remark}

\newtheorem*{question}{Question}
\newtheorem*{lemma-repeat}{Lemma \ref{5/12lemma}}

\theoremstyle{definition}
\newtheorem{definition}[proposition]{Definition}

\newcommand{\tb}{\textbf}

\newcommand{\vs}{\vspace{3mm}}
\newcommand{\vsss}{\vspace{6mm}}

\newcommand{\hs}{\hspace{1mm}}

\newcommand{\R}{\mathbb{R}}

\newcommand{\N}{\mathbb{N}}

\newcommand{\mc}{\mathcal}

\newcommand{\ep}{\epsilon}
\newcommand{\lam}{\lambda}
\newcommand{\sub}{\subseteq}

\DeclareMathOperator{\PPC}{PPC}

\DeclareMathOperator{\summ}{sum}
\DeclareMathOperator{\LHS}{LHS}
\DeclareMathOperator{\RHS}{RHS}
\DeclareMathOperator{\sinc}{sinc}

\title{On the Smallest Gap in a Sequence with Poisson Pair Correlations}
\author{Daniel Altman \and Zachary Chase}
\thanks{The second author is partially supported by Ben Green's Simons Investigator Grant 376201 and gratefully acknowledges the support of the Simons Foundation.}
\address{Mathematical Institute, Andrew Wiles Building, Radcliffe Observatory Quarter, Woodstock Road, Oxford OX2 6GG, UK}
\date{September 18, 2020}

\begin{document}

\begin{abstract}
We prove that any increasing sequence of real numbers with average gap $1$ and Poisson pair correlations has some gap that is at least $3/2+10^{-9}$. This improves upon a result of Aistleitner, Blomer, and Radziwi\l\l. 
\end{abstract}

\maketitle

\section{Introduction}\label{introduction}

Let $\lam = (\lam_n)_{n=1}^\infty$ be an increasing sequence of real numbers. Often, for number theoretic sequences $\lam$, the average gap $\lam_{n+1}-\lam_n$ is well-understood, while little is known about the distribution function of the gaps. Sometimes, however, statistical information about the collection of gaps $\lam_{n+k}-\lam_n$ is of importance.    

\vs

For example, letting $(\gamma_n)_{n=1}^\infty$ denote the imaginary parts of the zeroes of the Riemann zeta function in the critical strip in increasing order, we know $$\#\{\gamma_n \le T\} \sim \frac{T\log T}{2\pi}$$ as $T \to \infty$, and Montgomery's pair-correlation conjecture predicts that $$\frac{2\pi}{T\log T}\#\Big\{(n,m) : \gamma_n,\gamma_m \le T \text{ , } \frac{2\pi a}{\log T} \le \gamma_m-\gamma_n \le \frac{2\pi b}{\log T}\Big\} \to \int_a^b \left(1-\sinc^2(\pi t)\right)dt \vspace{2mm}$$ as $T \to \infty$, for any fixed $0 < a < b$, where $\sinc(x):= \frac{\sin x}{x}$. 

\vs

Henceforth, let $\lam = (\lam_n)_{n=1}^\infty$ denote an increasing sequence with average gap $1$: $$\frac{1}{N}\sum_{n \le N} (\lam_{n+1}-\lam_n) \to 1. \vspace{-1mm}$$ With this normalization, we may define the \textit{pair correlation function} $R_\lam$ by $$R_\lam(I,N) := \frac{1}{N}\#\Big\{(i,j) : 1 \le i \not = j \le N : \lambda_j-\lambda_i \in I\Big\},$$ where $I \sub \R$ is a bounded interval and $N$ a positive integer. 

\vs

For example, up to some normalization technicalities, Montgomery's pair-correlation conjecture asserts that $R_{(\gamma_n)_n}$ converges (in distribution) to a distribution with cumulative distribution function $1-\sinc^2(\pi t)$.   

\vs

Motivated by the fact that the pair correlation function of a random sequence generated by a Poisson point process converges in distribution to the uniform distribution, an increasing sequence of real numbers $(\lam_n)_{n=1}^\infty$ with average gap $1$ is said to have \textit{Poisson pair correlations} (PPC) if $R_{(\lam_n)_n}$ converges to the uniform distribution: \begin{equation}\label{ppcdefn}\lim_{N \to \infty} \frac{1}{N}\Big|\Big\{(i,j) : 1 \le i \not = j \le N : \lambda_j-\lambda_i \in I\Big\}\Big| = |I|\end{equation} for all intervals $I \subseteq \R$, where $|\cdot|$ denotes the Lebesgue measure. 

\vs

Despite Montgomery's pair-correlation conjecture concerning increasing sequences of real numbers, most research on properties of general sequences with Poisson pair correlation concerns sequences in the torus (see, for example, \cite{all}, \cite{alp}, \cite{heath-brown}, \cite{rudnicksarnak}, \cite{rz1999}, \cite{rz2002}, \cite{steinerberger}, \cite{walker}), with little investigated about sequences of real numbers.



\vs

Some specific number-theoretic sequences of real numbers have been shown to have PPC. Sarnak \cite{sarnak} showed that almost every positive definite binary quadratic form (in a suitable sense) gives rise to a sequence with PPC, by ordering the values it takes on (pairs of) positive integers and appropriately normalizing. Concretely, a consequence of the work of Eskin, Margulis, and Mozes \cite{emm} is that the ordered sequence of values of $x^2+\sqrt{2}y^2$ for $x,y \in \N$ has PPC. 

\vs

Aistleitner, Blomer, and Radziwi\l\l \hspace{0.5mm} \cite{abr} studied the related triple correlation function of certain number-theoretic sequences, while also initiating a study of general sequences of real numbers with Poisson pair (and triple) correlations. They asked the following.

\vspace{1.5mm}

\begin{question}
Let $\lam_1 < \lam_2 < \dots$ be an increasing sequence of real numbers with average gap $1$ and with Poisson pair correlations. How small can $\limsup_{n \to \infty} \lam_{n+1}-\lam_n$ be?
\end{question}

\vspace{1.5mm}

Among increasing sequences of real numbers with average gap $1$ and PPC, Aistleitner et al. exhibited one with maximum gap $2$, and proved that any such sequence must have a gap of size at least $3/2-\epsilon$, for any $\epsilon > 0$. They asked in their paper \cite{abr} as well as at Oberwolfach 2019 \cite{green} to improve either bound. Our main theorem is an improved lower bound. 

\vspace{2mm}

\begin{theorem}\label{main}
Let $\lam_1 < \lam_2 < \dots$ be an increasing sequence of real numbers with average gap $1$ and Poisson pair correlations. Then $\limsup_{n \to \infty} \lam_{n+1}-\lam_n > \frac{3}{2}+10^{-9}$. 
\end{theorem}

\vspace{1.5mm}

We leave open the question of how small the largest gap can be; in light of Theorem \ref{main}, it lies between $\frac{3}{2}+10^{-9}$ and $2$, inclusive.

\vs

\section{Motivation and Proof Sketch of Theorem \ref{main}}\label{motivation}

In this section we motivate the proof of Theorem \ref{main}, overlooking some technical complications and emphasizing the main ideas. 

\vs

Let us paraphrase the proof sketch given in \cite{abr} that any strictly increasing sequence of real numbers with mean gap $1$ and PPC has a gap at least $\frac{3}{2}-\ep$, for any $\ep > 0$. 

\vs

Indeed, for any sequence $(\lam_n)_n$ that has PPC, the distribution function of the gaps $\lam_{n+1}-\lam_n$ can grow at most linearly. However, if $\lam_{n+1}-\lam_n \le 3/2-\ep$ for each $n$, then the average spacing being $1$ is incompatible with the distribution function growing at most linearly between $1/2-\epsilon$ and $3/2-\epsilon$. 

\vs

This argument also implies that if the maximum gap is $3/2$, then with $F$ denoting the distribution function of the gaps $\lam_{n+1}-\lam_n$, we have \begin{equation}\label{F=0} F(x) = 0 \text{ for } x \le \frac{1}{2}\end{equation} \begin{equation}\label{F=x-1/2} F(x) = x-\frac{1}{2} \text{ for } x \in \big[\frac{1}{2},\frac{3}{2}\big].\end{equation}

\vs

We now mention the consequences of \eqref{F=0} and \eqref{F=x-1/2}. The latter means that $\PPC(\frac{1}{2},\frac{3}{2})$ (i.e., \eqref{ppcdefn} for $I = [\frac{1}{2},\frac{3}{2}]$) is already satisfied by the single gaps $\lam_{n+1}-\lam_n$, so there cannot be a nontrivial contribution coming from larger gaps $\lam_{n+m}-\lam_n$, $m \ge 2$. The former, \eqref{F=0}, on the other hand, means that $\PPC(0,\frac{1}{2})$ must come entirely from a $0$-density part of the sequence, and more specifically only from blocks $[n_1,n_2]$ contained in that $0$-density part. For the union of such blocks to nontrivially contribute to the $\PPC$ count, the blocks must grow in length. 

\vs

The natural question then is whether such blocks can satisfy the $\PPC$ condition on all subintervals of $[0,\frac{1}{2}]$. We show that the answer is no. The key is to establish a ``bias near $0$" of the $\PPC$ count on long blocks whose total gap is at most $1/2$. A bit more precisely, if $\lam_1 < \dots < \lam_k$ have $\lam_k -\lam_1 \le 1/2$, then $\frac{1}{|J|}\sum_{1 \le i < j \le k} 1_{\lam_j-\lam_i \in J}$ is larger for intervals $J \sub [0,1/2]$ concentrated near $0$, with the bias becoming more pronounced as $k \to \infty$. A difficulty we encounter in the proof, though, is that the blocks forming the relevant $0$-density part of the sequence need not have total gap at most $1/2$, so we need to use a further decomposition of the $0$-density part of the sequence. We end up using a suitable greedy algorithm to decompose; this is implemented in Section \ref{partitionsection}.

\vs

\section{Proof of Theorem \ref{main}}\label{proofofmaintheorem}

Let $\ep = 10^{-9}$. For this section, we fix a supposed increasing sequence of real numbers $\lam_1 < \lam_2 < \dots$ with average gap $1$ and $\PPC$, that has $\lam_{n+1}-\lam_n \le 3/2+\ep$ for sufficiently large $n$. By truncating the sequence, we may assume that $$g_n := \lambda_{n+1}-\lam_n$$ satisfies $g_n \le 3/2+\ep$ for each $n \ge 1$. As in Section \ref{motivation}, we write $\PPC(a,b)$ to denote equation \eqref{ppcdefn} for $I = [a,b]$. We note that a sequence satisfying the PPC condition for all such $I$ necessarily satisfies the same condition for all open or indeed half-open intervals. We may therefore also use $\PPC(a,b)$ to refer to equation \eqref{ppcdefn} for the half-open interval $[a,b)$, for example.

\vs

Recall that the lower bound of $\frac{3}{2}$ was established in \cite{abr}, whose proof we sketched in Section \ref{motivation}. We start the proof of Theorem \ref{main} by making these arguments quantitative.

\vs

We begin with equation \eqref{F=0}.

\vspace{1.5mm}

\begin{proposition}\label{smalldensity}
For all $N$ sufficiently large, we have $$\frac{1}{N}\#\{n \le N : g_n \le \frac{1}{2}\} \le 2\sqrt{\epsilon}.$$
\end{proposition}

\begin{proof} 
First, note that \begin{align}\label{smalldensity1} \frac{1}{N}\sum_{n \le N} g_n &= \int_0^{\frac{3}{2}+\epsilon} \frac{1}{N}\#\{n \le N : g_n > x\}dx \\ \nonumber &= \int_0^{\frac{1}{2}+\sqrt{\epsilon}} \frac{1}{N}\#\{n \le N : g_n > x\}dx + \int_{\frac{1}{2}+\sqrt{\epsilon}}^{\frac{3}{2}+\epsilon} \frac{1}{N}\#\{n \le N : g_n > x\}dx \\ \nonumber &= \frac{1}{2}+\sqrt{\epsilon}-\int_0^{\frac{1}{2}+\sqrt{\epsilon}} \frac{1}{N}\#\{n \le N : g_n \le x\}dx \\ \nonumber &\qquad \qquad \qquad \qquad \qquad \qquad + \int_{\frac{1}{2}+\sqrt{\epsilon}}^{\frac{3}{2}+\epsilon} \frac{1}{N}\#\{n \le N : g_n \in (x,\frac{3}{2}+\epsilon)\}dx. \end{align} Now, we claim that \begin{equation}\label{smalldensity2} \limsup_{N \to \infty} \int_{\frac{1}{2}+\sqrt{\epsilon}}^{\frac{3}{2}+\epsilon} \frac{1}{N}\#\{n \le N : g_n \in (x,\frac{3}{2}+\epsilon)\}dx \le \int_{\frac{1}{2}+\sqrt{\ep}}^{\frac{3}{2}+\ep} (\frac{3}{2}+\ep-x)dx.\end{equation} Indeed, the pointwise upper bound $$\frac{1}{N}\#\{n \le N : g_n \in (x,\frac{3}{2}+\epsilon)\} \le \min\Big(1,\frac{1}{N}\sum_{n \le N} \sum_{m \le N-n+1} 1_{g_n+\dots+g_{n+m-1} \in (x,\frac{3}{2}+\epsilon)}\Big)$$ together with the dominated convergence theorem and $\PPC(x, \frac{3}{2}+\ep)$ for $x \in (\frac{1}{2}+\sqrt{\ep},\frac{3}{2}+\ep)$, namely, $$\lim_{N \to \infty} \frac{1}{N}\sum_{n \le N} \sum_{m \le N-n+1} 1_{g_n+\dots+g_{n+m-1} \in (x,\frac{3}{2}+\epsilon)} = \frac{3}{2}+\epsilon-x,$$ gives \eqref{smalldensity2}. That the average gap of $(\lam_n)_n$ is $1$ corresponds to \begin{equation}\label{avggap1} \lim_{N \to \infty} \frac{1}{N}\sum_{n \le N} g_n = 1. \end{equation} Rearranging \eqref{smalldensity1}, taking $N \to \infty$, and using \eqref{smalldensity2} and \eqref{avggap1} gives \begin{align*} \limsup_{N \to \infty} \int_0^{\frac{1}{2}+\sqrt{\epsilon}} \frac{1}{N}\#\{n \le N : g_n \le x\}dx &\le -1+\frac{1}{2}+\sqrt{\epsilon}+\int_{\frac{1}{2}+\sqrt{\epsilon}}^{\frac{3}{2}+\epsilon} (\frac{3}{2}+\epsilon-x)dx \\ &= \frac{3}{2}\epsilon+\frac{1}{2}\epsilon^2-\epsilon^{3/2}.\end{align*} Thus, using the trivial $$\frac{1}{N} \#\{n \le N : g_n \le x\} \ge \frac{1}{N}\#\{n \le N : g_n \le \frac{1}{2}\}$$ for $x \in [\frac{1}{2},\frac{1}{2}+\sqrt{\ep}]$ and the even more trivial lower bound of $0$ when $x \in [0,\frac{1}{2}]$, yields $$\limsup_{N \to \infty} \sqrt{\epsilon}\cdot \frac{1}{N}\#\{n \le N : g_n \le \frac{1}{2}\} \le \frac{3}{2}\epsilon+\frac{1}{2}\epsilon^2-\epsilon^{3/2}.$$ Dividing by $\sqrt{\ep}$, Proposition \ref{smalldensity} follows.
\end{proof}

\vs

We now use the quantitative version of \eqref{F=x-1/2} to argue that the PPC($\frac{1}{2},\frac{3}{2}+\ep$) contribution comes nearly entirely from single gaps $g_n$. 

\vspace{1.5mm}

\begin{proposition}\label{smallPPCmatleast2}
For all $N$ sufficiently large, we have
$$\frac{1}{N} \sum_{n \le N} \sum_{2 \le m \le N-n+1} 1_{g_n+\dots+g_{n+m-1} \in (\frac{1}{2},\frac{3}{2}+\epsilon)} \le 2\sqrt{\epsilon}.$$
\end{proposition}

\begin{proof}
As in the proof of Proposition \ref{smalldensity}, $$\frac{1}{N}\sum_{n \le N} g_n = \frac{1}{2}-\sqrt{\epsilon}-\int_0^{\frac{1}{2}-\sqrt{\epsilon}} \frac{1}{N}\#\{n \le N : g_n \le x\}dx $$ $$\qquad \qquad \qquad \qquad \qquad \qquad \qquad\qquad \qquad +  \int_{\frac{1}{2}-\sqrt{\epsilon}}^{\frac{3}{2}+\epsilon} \frac{1}{N}\#\{n \le N : g_n \in (x,\frac{3}{2}+\epsilon)\}dx,$$ which, by merely dropping a (negative) term, gives \begin{equation}\label{matleast21}\frac{1}{N}\sum_{n \le N} g_n \le \frac{1}{2}-\sqrt{\epsilon}+\int_{\frac{1}{2}-\sqrt{\epsilon}}^{\frac{3}{2}+\epsilon} \frac{1}{N}\#\{n \le N : g_n \in (x,\frac{3}{2}+\epsilon)\}dx.\end{equation} We write $$\frac{1}{N}\#\{n \le N : g_n \in (x,\frac{3}{2}+\ep)\} = \frac{1}{N}\sum_{n \le N}\sum_{m \le N-n+1} 1_{g_n+\dots+g_{n+m-1} \in (x,\frac{3}{2}+\ep)}$$ $$\hspace{70mm} -\frac{1}{N}\sum_{n \le N}\sum_{2 \le m \le N-n+1} 1_{g_n+\dots+g_{n+m-1} \in  (x,\frac{3}{2}+\ep)}$$ and use the same dominated convergence theorem argument as in the proof of Proposition \ref{smalldensity} to obtain, from \eqref{matleast21}, that \begin{align*} &\limsup_{N \to \infty} \int_{\frac{1}{2}-\sqrt{\epsilon}}^{\frac{3}{2}+\epsilon} \frac{1}{N}\sum_{n \leq N}\sum_{2 \le m \le N-n+1} 1_{g_n+\dots+g_{n+m-1} \in (x,\frac{3}{2}+\epsilon)}dx \\ &\qquad \qquad \le -1+\frac{1}{2}-\sqrt{\epsilon}+\int_{\frac{1}{2}-\sqrt{\epsilon}}^{\frac{3}{2}+\epsilon} (\frac{3}{2}+\epsilon-x)dx \\ &\qquad \qquad = \frac{3}{2}\epsilon+\frac{1}{2}\epsilon^2+\epsilon^{3/2},\end{align*} and thus $$\limsup_{N \to \infty} \sqrt{\epsilon}\cdot \frac{1}{N}\sum_{n \leq N}\sum_{2 \le m \le N-n+1} 1_{g_n+\dots+g_{n+m-1} \in (\frac{1}{2},\frac{3}{2}+\epsilon)} \le \frac{3}{2}\epsilon+\frac{1}{2}\epsilon^2+\epsilon^{3/2}.$$ Dividing by $\sqrt{\ep}$, the proposition follows.
\end{proof}

\vs

We now exploit the aforementioned ``bias" towards $0$ exhibited by large intervals with sum of gaps at most $1/2$. We will need a technical lemma, proven in the appendix but assumed for now.

\vs

\begin{lemma}\label{5/12lemma}
For positive integers $a,b,c,L$ satisfying $1 \le a \le b \le c \le L$, we have  $$(a-1)a+(b-a)(b-a+1)+(c-b)(c-b+1)+(L-c)(L-c+1) \qquad \qquad$$ $$\qquad \qquad \qquad \qquad +(a-1)(b-a)+(b-a)(c-b)+(c-b)(L-c) \ge \frac{5}{12}L^2+\frac{1}{6}L-\frac{7}{12}.$$
\end{lemma}

\vs

Lemma \ref{5/12lemma} allows us to show that, instead of getting the desired $\frac{1}{4}+\frac{1}{8} = \frac{3}{8}$ for $\PPC(0,\frac{1}{4})+\PPC(0,\frac{1}{8})$, we get at least $\frac{5}{12} = \frac{3}{8}+\frac{1}{24}$, asymptotically for large intervals.

\vspace{1mm}

\begin{proposition}\label{biasnear0}
Let $L \ge 1$ be a positive integer and $g_1,\dots,g_L$ be positive reals with $\sum_{i=1}^L g_i \le \frac{1}{2}$. Then, $$\sum_{n \le L}\sum_{m \le L-n+1} 1_{g_n+\dots+g_{n+m-1} \le \frac{2}{8}}+\sum_{n \le L}\sum_{m \le L-n+1} 1_{g_n+\dots+g_{n+m-1} \le \frac{1}{8}} \ge \frac{5}{6}\binom{L+1}{2} - \frac{5}{6}L.$$
\end{proposition}

\begin{proof}
By scaling, it suffices to prove the proposition when $\sum_{i=1}^L g_i = \frac{1}{2}$. Suppose $\sum_{i=1}^L g_i = \frac{1}{2}$. Let $$a = \min\{j \le L : g_1+\dots+g_j \ge \frac{1}{8}\},$$ $$b = \min\{j \le L : g_1+\dots+g_j \ge \frac{2}{8}\},$$ $$c = \min\{j \le L : g_1+\dots+g_j \ge \frac{3}{8}\},$$ and note $$\sum_{n \le L}\sum_{m \le L-n+1} 1_{g_n+\dots+g_{n+m-1} \le \frac{1}{8}} \ge \frac{(a-1)a}{2}+\frac{(b-a)(b-a+1)}{2}+\frac{(c-b)(c-b+1)}{2}$$ $$\hspace{30mm} +\frac{(L-c+1)(L-c+2)}{2}.$$ by doing casework in which intervals $n$ and $n+m-1$ lie (the different intervals are $[1,a),[a,b),[b,c),[c,L]$). Similarly, $$\sum_{n \le L} \sum_{m \le L-n+1} 1_{g_n+\dots+g_{n+m-1} \le \frac{2}{8}} \ge \frac{(a-1)a}{2}+\frac{(b-a)(b-a+1)}{2}+\frac{(c-b)(c-b+1)}{2}$$ $$\hspace{60mm} +\frac{(L-c+1)(L-c+2)}{2}+(a-1)(b-a)$$ $$\hspace{52.7mm} +(b-a)(c-b)+(c-b)(L-c+1),$$ Therefore, $$\sum_{n \le L} \sum_{m \le L-n+1} 1_{g_n+\dots+g_{n+m-1} \le \frac{2}{8}} +  \sum_{n \le L}\sum_{m \le L-n+1} 1_{g_n+\dots+g_{n+m-1} \le \frac{1}{8}}$$ $$\hspace{40mm} \ge (a-1)a+(b-a)(b-a+1)+(c-b)(c-b+1)+(L-c+1)(L-c+2)$$ $$\hspace{26.5mm} +(a-1)(b-a)+(b-a)(c-b)+(c-b)(L-c+1).$$ Lower bounding $L-c+1$ and $L-c+2$ by $L-c$ and $L-c+1$, respectively, Lemma \ref{5/12lemma} finishes the proof of Proposition \ref{biasnear0}, since $\frac{5}{12}L^2+\frac{1}{6}L-\frac{7}{12} \ge \frac{5}{6}{L + 1 \choose 2}-\frac{5}{6}L$ for $L \ge 1$. 
\end{proof}

\vs

We now proceed to isolate the relevant ``$0$-density" parts of the sequence on which the gaps are at most $1/2$, in order to exploit the bias that Proposition \ref{biasnear0} illustrates. 

\vspace{4.5mm}

Here and henceforth, we let $[N]:=\{1,2,\ldots,N\}$.

\begin{definition}
For a nonempty interval $J \sub [N]$, let $L(J),R(J)$ denote the left and right endpoints of $J$, respectively, and let $\summ(J) = \sum_{n \in J} g_n$. 
\end{definition}

\begin{definition}
For a interval $J \sub [N]$, we denote \vspace{0.5mm} $$\PPC^J(0,a) := \sum_{n \le n' \in J} 1_{g_n+\dots+g_{n'} < a}.$$ For intervals $J_1, J_2 \sub [N]$, with $R(J_1) < L(J_2)$, we denote \vspace{1mm} $$\PPC^{J_1,J_2}(0,a) := \sum_{(n,n') \in J_1\times J_2} 1_{g_n+\dots+g_{n'} < a}.$$ 
\end{definition}

\vspace{1mm}

Take a large $N$. Let $\mc{I}^N$ denote the collection of all maximal intervals on which the gaps $g_n$ are at most $\frac{1}{2}$. More formally, we define $\mc{I}^N$ to be the collection of all intervals $I \sub [N]$ such that (1) $g_n \le \frac{1}{2}$ for each $n \in I$, (2) $L(I) = 1$ or $g_{L(I)-1} > \frac{1}{2}$, and (3) $R(I) = N$ or $g_{R(I)+1} > \frac{1}{2}$.

\vs

We begin by noting the following. 

\begin{lemma} \label{Ibounds}
For all large $N$, we have 
\[\sum_{I \in \mc I^N} |I| \leq 2 \sqrt{\epsilon} N.\]
Furthermore, as $N \to \infty$, we have
\[ \sum_{I \in \mc I^N} \binom{|I|+1}{2} \geq N(\frac{1}{2}+o(1)).\]
\end{lemma} 
\begin{proof}
By definition we have $$\sum_{I \in \mc I^N} |I| = \frac{1}{N} \# \{n \leq N : g_n \leq \frac{1}{2}\}.$$ Proposition \ref{smalldensity} then gives the first inequality. For the second inequality, note, by the maximality of the intervals comprising $\mc{I}^N$, that
\[ (\frac{1}{2} +o(1))N = \PPC^{[N]}(0,\frac{1}{2}) = \sum_{I \in \mc{I}^N}\PPC^I(0,\frac{1}{2}) \leq \sum_{I \in \mc{I}^N} \binom{|I|+1}{2}.\] 
\end{proof}

We provide a quick remark on motivation. 

\begin{remark}\label{remarker}
Observe that, if it were the case that $\summ(I) \le \frac{1}{2}$ for each $I \in \mc{I}^N$, then we could conclude the proof of Theorem \ref{main} as follows. By Proposition \ref{biasnear0}, one has 
\begin{align*}
(\frac{3}{8} + o(1))N &= \PPC^{[N]}(0,\frac{1}{8}) + \PPC^{[N]}(0,\frac{1}{4}) \\ &\geq \frac{5}{6}\sum_I \binom{|I|+1}{2} - \frac{5}{6}\sum_I |I| \\ &\geq   (\frac{5}{12} - \frac{5}{3}\sqrt \epsilon+o(1))N
\end{align*}
as $N \to \infty$, which would give our desired contradiction by taking $N$ sufficiently large. 
\end{remark}

However, it need not be the case that $\summ(I) \le \frac{1}{2}$ for each $I \in \mc{I}^N$. In light of Remark \ref{remarker}, therefore, the strategy is to partition each $I \in \mc{I}^N$ into subintervals $\{J_k^I\}_k$ with $\summ(J_k^I) \leq 1/2$ for each $k$ and such that, for all $a\in(0,1/2]$ and all $k$, the contribution to PPC$(0,a)$ from windows that overlap with $J_k^I$ comes nearly entirely from windows that lie entirely inside $J_k^I$. The existence of such a partition is not at all immediate. We provide in Proposition \ref{partition} below a precise statement of what is needed. 

\begin{proposition}\label{partition}
Let $\mc{I}^N$ be as above. There exists a partition of each $I \in \mc{I}^N$ into subintervals $\{J_k^I\}_{k=1}^{r_I}$ such that the following two hold.

\vspace{1.5mm}

\begin{enumerate}
\item $\summ(J_k^I)\leq \frac{1}{2}$ for each $I \in \mc{I}^N$ and $k \in \{1,\dots,r_I\}$, and 

\vspace{1.5mm}

\item $\sum_{I \in \mc I^N}\sum_{k=1}^{r_I} \binom{|J^I_k|+1}{2} \geq  (\frac{1}{2} - 4\sqrt{2}\epsilon^{1/4}) N$.
\end{enumerate}
\end{proposition}

\vspace{1.5mm}

To quickly conclude the proof of Theorem \ref{main}, we postpone the proof of Proposition \ref{partition} (and the description of the partition) to the following section, and assume it for now. 

%
%

\vs

\begin{proof}[Proof of Theorem \ref{main}]

Proceeding along the lines of Remark \ref{remarker}, Proposition \ref{biasnear0}, Proposition \ref{partition}, and Lemma \ref{Ibounds} yield 
\begin{align*} 
(\frac{3}{8}+o(1))N 
&= \PPC^{[N]}(0,\frac{1}{8})+\PPC^{[N]}(0,\frac{1}{4}) \\ 
&\ge \sum_{I \in \mc{I}^N} \sum_k \PPC^{J_k^I}(0,\frac{1}{8})+\PPC^{J_k^I}(0,\frac{1}{4})\\ 
&\ge \frac{5}{6}\sum_{I \in \mc{I}^N} \sum_k \binom{|J_k^I|+1}{2} - \frac{5}{6}\sum_{I \in \mc{I}^N} \sum_k |J_k^I|  \\ 
&\ge \frac{5}{6}(\frac{1}{2}-4\sqrt{2}\epsilon^{1/4})N-\frac{5}{3}\sqrt{\epsilon}N.
\end{align*} Rearranging, dividing by $N$, and sending $N \to \infty$, we obtain $$\frac{10\sqrt{2}}{3}\epsilon^{1/4} + \frac{5}{3}\sqrt \epsilon - \frac{1}{24} \ge 0,$$ which is indeed false for $\epsilon = 10^{-9}$ (but not for $\epsilon = 10^{-8}$). This gives the desired contradiction to our assumption that a sequence $(\lam_n)_n$ with PPC, average gap $1$, and maximum gap $3/2+\ep$ exists.
\end{proof}


\vspace{1.5mm}

\section{Partitioning, and a proof of Proposition \ref{partition}}\label{partitionsection}

{\normalfont 

\vspace{1.5mm}

The following examples are helpful to keep in mind to explain the need for care when choosing the partition of a given $I \in \mc{I}^N$, and to help motivate the partition we will use. \vspace{1.5mm} $${\textstyle I_1 = \left\{\frac{2}{5},0,0,0,\dots,0,0,0,\frac{1}{3},0,0,0,\dots,0,0,0,\frac{2}{5}\right\}}$$ \vspace{1.5mm} $${\textstyle I_2 = \left\{\frac{1}{4},0,0,0,\dots,0,0,0,\frac{1}{2},\frac{1}{2},\frac{1}{2},\frac{1}{2},\frac{1}{2},0,0,0,\dots,0,0,0,\frac{1}{4}\right\}}$$ \vspace{1.5mm} $${\textstyle I_3 = \left\{0,0,\dots,0,0,\frac{1}{3},0,0,\dots,0,0,\frac{1}{3},0,0,\dots,0,0\right\}.}$$ 

\vs

The first example $I_1$ shows that a ``greedy division", in which one goes from left to right, dividing immediately before the sum first exceeds $1/2$, will not work. Indeed, that division is $${\textstyle \left\{\frac{2}{5},0,0,0,\dots,0,0,0\right\},\left\{\frac{1}{3},0,0,0,\dots,0,0,0\right\},\left\{\frac{2}{5}\right\}},$$ which is problematic as there will be much contribution to the PPC($0,\frac{1}{2}$) count coming from different subintervals; specifically, any index besides the first in the first subinterval and any index in the second subinterval would prove a nontrivial contribution. 

\vspace{1mm}

Similarly, another ``greedy division" in which the largest numbers successively ``claim" the largest subinterval they can, will not work. For the case of $I_1$, the division is $${\textstyle \left\{\frac{2}{5},0,0,0,\dots,0,0,0\right\},\left\{\frac{1}{3}\right\},\left\{0,0,0,\dots,0,0,0,\frac{2}{5}\right\},}$$ which has a large contribution to $\PPC(0,\frac{1}{2})$ coming from any index in the first subinterval besides the first and any index in the last subinterval besides the last. Note thus that even ``non-adjacent" subintervals can cause issues.

\vspace{1mm}

A division that does work for $I_1$ is $${\textstyle \left\{\frac{2}{5}\right\},\left\{0,0,0,\dots,0,0,0,\frac{1}{3},0,0,0,\dots,0,0,0\right\},\left\{\frac{2}{5}\right\},} $$ as there is only a minor contribution (namely, linear in the size of the interval rather than quadratic) coming from different subintervals. 

\vs

For $I_2$, essentially any reasonable division is permissible, but we draw attention to it as it shows that sometimes the \textit{reason} for negligible contribution from different subintervals are subintervals in between. For example, if we decompose as $${\textstyle \left\{\frac{1}{4},0,0,0,\dots,0,0,0\right\},\left\{\frac{1}{2}\right\},\left\{\frac{1}{2}\right\},\left\{\frac{1}{2}\right\},\left\{\frac{1}{2}\right\},\left\{\frac{1}{2}\right\},\left\{0,0,0,\dots,0,0,0,\frac{1}{4}\right\},}$$ then the reason that there is no contribution to $\PPC(0,\frac{1}{2})$ from the subintervals $\{\frac{1}{4},0,0,\dots,0,0\}$ and $\{0,0,\dots,0,0,\frac{1}{4}\}$ are the five subintervals $\{\frac{1}{2}\}$ in between. 

\vs

For $I_3$, any division will admit a large PPC contribution from different subintervals. This would of course be harmful, but we make use of the fact that it won't exist often in our situation, since it provides a nontrivial contribution to $\PPC(\frac{1}{2},1)$ (which we already know comes nearly entirely from single gaps).} 

\vsss

With the above examples in mind, we now choose the partition we use, to prove Proposition \ref{partition}. 

\vs

Fix $I \in \mc{I}^N$. In the following definition, ties may be broken arbitrarily. Let $J_1$ be the largest subinterval of $I$ with $\summ(J_1) \le \frac{1}{2}$. With $J_1,\dots,J_r$ already defined, if $\cup_{k=1}^r I_k \not = I$, let $J_{r+1}$ be the largest subinterval of $I\setminus\cup_{k=1}^r J_k$ with $\summ(J_{r+1}) \le \frac{1}{2}$. Let $J_1,\dots,J_s$ be all the subintervals resulting from this process. Of course $s \le |I| < +\infty$. 

\vs

Clearly $I = \sqcup_{k=1}^s J_k$ and $\summ(J_k) \leq 1/2$ for each $k$, establishing the first requirement of Proposition \ref{partition}. We now begin to proceed to establish the second.

\vs

Hopefully not confusing the reader, we renumber now so that $J_1$ is the leftmost interval, with $J_2$ to the immediate right of $J_1$, $J_3$ to the immediate right of $J_2$, etc.. For $1 \le k \le s-1$, let $g_1(k) \in \{k,k+1\}$ and $b_1(k) \in \{k,k+1\}\setminus \{g_1(k)\}$ be such that $J_{g_1(k)}$ was chosen before $J_{b_1(k)}$. Note, in particular, that $|J_{g_1(k)}| \ge |J_{b_1(k)}|$. 

\vs

We quickly pin down exactly which different subintervals need to be considered with regards to their contribution to $\PPC(0,a)$, for $a \leq 1/2$. 

\vspace{1.5mm}

\begin{definition}
Call $k \in [2,s-1]$ \textit{sandwiched} if it was chosen after each of its neighboring subintervals, i.e., if $b_1(k-1) = k$ and $b_1(k) = k$. For a sandwiched $k$, let $g_2(k) \in \{k-1,k+1\}$ and $b_2(k) \in \{k-1,k+1\}\setminus\{g_2(k)\}$ be such that $J_{g_2(k)}$ was chosen before $J_{b_2(k)}$. In particular, $|J_{g_2(k)}| \ge |J_{b_2(k)}|$. 
\end{definition}

\vspace{1.5mm}

\begin{lemma}\label{trichotomy}
If $n \le n' \in I$ have $g_n+\dots+g_{n'} \le \frac{1}{2}$, then either \begin{enumerate} \item $(n,n') \in J_k\times J_k$ for some $k$, \item $(n,n') \in J_k\times J_{k+1}$ for some $k$, or \item $(n,n') \in J_{k-1}\times J_{k+1}$ for some sandwiched $k$.\end{enumerate}
\end{lemma}

\begin{proof}
First note that if $n \in J_k$ and $n' \in \cup_{\Delta \ge 3} J_{k+\Delta}$, then $g_n+\dots+g_{n'} > \frac{1}{2}$, since $\summ(J_{k+1}\cup J_{k+2}) > \frac{1}{2}$ (for otherwise whichever of $J_{k+1},J_{k+2}$ was chosen first would have ``engulfed" the other). Now suppose $n \in J_{k-1}$ and $n' \in J_{k+1}$ for some $k$. If $b_1(k-1) = k-1$, then $\summ(R(J_{k-1})\cup J_k) > \frac{1}{2}$, so $g_n+\dots+g_{n'} > \frac{1}{2}$. Similarly, if $b_1(k) = k+1$, then $\summ(J_k\cup L(J_{k+1})) > \frac{1}{2}$ also yields $g_n+\dots+g_{n'} > \frac{1}{2}$. Hence, $k$ is sandwiched. \end{proof}

\vs

We now proceed to argue that the PPC$(0,\frac{1}{2})$ contribution coming from cases (2) or (3) in Lemma \ref{trichotomy} is small. We begin with case (2).

\vs

We shall argue that the $\PPC(0,\frac{1}{2})$ contribution coming from adjacent subintervals $J_k,J_{k+1}$ is small by arguing that $|J_{b_1(k)}|$ is (usually) small. We do this by arguing that we would otherwise have too large of a contribution to $\PPC(\frac{1}{2},\frac{3}{2})$ coming from gaps $\lambda_{n+m}-\lambda_n$ with $m \ge 2$ (contradicting Proposition \ref{smallPPCmatleast2}). 

\vspace{1.5mm}

\begin{proposition}\label{smallbad1interval}
For any $k \in [s-1]$, one has $$\sum_{(n,n') \in J_k\times J_{k+1}} 1_{g_n+\dots+g_{n'} > \frac{1}{2}} \ge \frac{1}{2}|J_{b_1(k)}|^2.$$ 
\end{proposition}

\begin{proof}
Without loss of generality, by symmetry we may assume $b_1(k) = k+1$.

\vs

For $y \in J_{k+1}$ and $x \in J_k$, if $y-x+1 > |J_k|$, then $g_x+\dots+g_y > \frac{1}{2}$, since otherwise $[x,y]$ would have been chosen as $J_k$ (in the greedy process defining the partition) instead of $J_k$. Therefore, \begin{align*} \sum_{(n,n') \in J_k\times J_{k+1}} 1_{g_n+\dots+g_{n'} > \frac{1}{2}} &\ge \sum_{y=L(J_{k+1})}^{R(J_{k+1})} \sum_{x = L(J_k)}^{y-R(J_k)+L(J_k)-1} 1 \\ &= (\frac{L(J_{k+1})+R(J_{k+1})}{2})|J_{k+1}|-R(J_k)|J_{k+1}| \\ &= \frac{1}{2}|J_{k+1}|^2+\frac{1}{2}|J_{k+1}|, \end{align*} with the last equality using $R(J_k) = L(J_{k+1})-1$. 
\end{proof}

\vspace{1mm}

Next we proceed to bound the contribution from intervals $J_{k-1}$, $J_{k+1}$ for $k$ sandwiched. For such $k$, we argue that $|J_{b_2(k)}|$ is (usually) small.

\vspace{1.5mm}

\begin{proposition}\label{smallbad2interval}
For a sandwiched $k$, one has $$\sum_{(n,n') \in J_{k-1}\times J_{k+1}} 1_{g_n+\dots+g_{n'} > \frac{1}{2}} \ge \frac{1}{2}|J_{b_2(k)}|^2.$$ 
\end{proposition}

\begin{proof}
Without loss of generality, by symmetry we may assume $b_2(k) = k+1$.

\vspace{1.5mm}

For $y \in J_{k+1}$ and $x \in J_{k-1}$, if $y-x+1 > |J_{k-1}|$, then $g_x+\dots+g_y > \frac{1}{2}$, since otherwise $[x,y]$ would have been chosen instead of $J_{k-1}$ (recall $J_k$ was also chosen after $J_{k-1}$, since $k$ is sandwiched). Therefore, $$\sum_{(n,n') \in J_{k-1}\times J_{k+1}} 1_{g_n+\dots+g_{n'} > \frac{1}{2}} \ge \sum_{y=L(J_{k+1})}^{R(J_{k+1})} \sum_{x = L(J_{k-1})}^{\min(R(J_{k-1}),y-|J_{k-1}|)} 1.$$ If $R(J_{k+1})-|J_{k-1}| \le R(J_{k-1})$, then we obtain \begin{align*} \sum_{(n,n') \in J_{k-1}\times J_{k+1}} 1_{g_n+\dots+g_{n'} > \frac{1}{2}} &\ge \sum_{y = L(J_{k+1})}^{R(J_{k+1})} \sum_{x = L(J_{k-1})}^{y-|J_{k-1}|} 1 \\ &= |J_{k+1}|\left(\frac{R(J_{k+1})+L(J_{k+1})}{2}-|J_{k-1}|-L(J_{k-1})+1\right) \\ &= |J_{k+1}|\left(\frac{R(J_{k+1})+L(J_{k+1})}{2}-R(J_{k-1})\right) \\ &= |J_{k+1}|\left(\frac{R(J_{k+1})+L(J_{k+1})}{2}-L(J_{k+1})+|J_k|+1\right) \\ &= |J_{k+1}|\left(\frac{|J_{k+1}|-1}{2}+|J_k|+1\right), \end{align*} and conclude by observing that $|J_k| \ge 0$ and $\frac{|J_{k+1}|-1}{2}+1 \ge \frac{1}{2}|J_{k+1}|$.

\vspace{2mm}

\noindent If, instead, $R(J_{k+1})-|J_{k-1}| \ge R(J_{k-1})$, we obtain \begin{equation}\label{obtain} \sum_{(n,n') \in J_{k-1}\times J_{k+1}} 1_{g_n+\dots+g_{n'} > \frac{1}{2}} \ge \sum_{y=L(J_{k+1})}^{|J_{k-1}|+R(J_{k-1})} \sum_{x=L(J_{k-1})}^{y-|J_{k-1}|} 1 +\sum_{y=|J_{k-1}|+R(J_{k-1})+1}^{R(J_{k+1})} \sum_{x=L(J_{k-1})}^{R(J_{k-1})} 1.\end{equation} The first double sum on the RHS of \eqref{obtain} is equal to $$(|J_{k-1}|+R(J_{k-1})-L(J_{k+1})+1)\left(\frac{|J_{k-1}|+R(J_{k-1})+L(J_{k+1})}{2}-R(J_{k-1})\right),$$ while the second double sum is equal to $$\Big(R(J_{k+1})-|J_{k-1}|-R(J_{k-1})\Big)|J_{k-1}|.$$ Adding these two sums and simplifying we obtain, 
$$(|J_{k-1}| - |J_k|)(\frac{|J_{k-1}| + |J_k| + 1}{2}) + |J_{k-1}|(|J_k| + |J_{k+1}|-|J_{k-1}|)\qquad \qquad \qquad$$ $$\qquad \qquad \qquad \qquad \qquad \qquad  = -\frac{1}{2}(|J_{k-1}|-|J_k|)^2 + \frac{1}{2}(|J_{k-1}| - |J_k|) + |J_{k-1}||J_{k+1}|.$$
Now recall that we have $|J_k| \leq |J_{k+1}| \leq |J_{k-1}|$ and furthermore, by our assumption in the second case, we have $|J_{k-1}| \leq |J_k| + |J_{k+1}|$. Thus we may obtain 
$$ -\frac{1}{2}(|J_{k-1}|-|J_k|)^2 + \frac{1}{2}(|J_{k-1}| - |J_k|) + |J_{k-1}||J_{k+1}| \geq -\frac{1}{2}|J_{k+1}|^2 + 0 + |J_{k+1}|^2 = \frac{1}{2}|J_{k+1}|^2.$$
This completes the proof. 
\end{proof}

\vspace{1mm}

\vspace{1mm}

\vs

We now cease referring to a specific $I \in \mc{I}^N$. To denote dependence on $I \in \mc{I}^N$, we denote $I = \sqcup_{k=1}^r J_k^I$ its decomposition. 

\vspace{1.5mm}

\begin{proposition}\label{PPCcrosstermbounds}
For any $a \in (0,\frac{1}{2})$, it holds that $$\sum_{I \in \mc{I}^N} \sum_k \PPC^{J_k^I,J_{k+1}^I}(0,a) \le 2\sqrt{2}\epsilon^{1/4}\left(\sum_{I \in \mc{I}^N} \sum_k |J_k^I|^2\right)^{1/2}\sqrt{N}$$ and $$\sum_{I \in \mc{I}^N} \sum_{k {\normalfont \text{ sandwiched}}} \PPC^{J_{k-1}^I,J_{k+1}^I}(0,a) \le 2\sqrt{2}\epsilon^{1/4}\left(\sum_{I \in \mc{I}^N} \sum_k |J_k^I|^2\right)^{1/2}\sqrt{N}.$$
\end{proposition}

\begin{proof}
By trivially bounding $\PPC^{J_k^I,J_{k+1}^I}(0,a) \le |J_{g_1(k)}^I|\hs |J_{b_1(k)}^I|$ and Cauchy-Schwarz, \begin{align*} \sum_{I \in \mc{I}^N} \sum_k \PPC^{J_k^I,J_{k+1}^I}(0,a) &\le \sum_{I \in \mc{I}^N} \sum_k |J_{g_1(k)}^I|\hs |J_{b_1(k)}^I| \\ &\le \left(\sum_{I,k} |J_{g_1(k)}^I|^2\right)^{1/2}\left(\sum_{I,k} |J_{b_1(k)}^I|^2\right)^{1/2} \\ &\le \left(2\sum_{I \in \mc{I}^N} \sum_k |J_k^I|^2\right)^{1/2}\left(\sum_{I \in \mc{I}^N} \sum_k 2 \PPC^{J_k^I,J_{k+1}^I}(\frac{1}{2},\frac{3}{2})\right)^{1/2},\end{align*} where the last inequality used Proposition \ref{smallbad1interval} together with the fact that $\summ(J_k^I \cup J_{k+1}^I) \le 1 \le \frac{3}{2}$. Now just observe $$\sum_{I \in \mc{I}^N} \sum_k 2 \PPC^{J_k^I,J_{k+1}^I}(\frac{1}{2},\frac{3}{2}) \le 2\sum_{n \le N} \sum_{2 \le m \le N-n+1} 1_{g_n+\dots+g_{n+m-1} \in (\frac{1}{2},\frac{3}{2})},$$ which, by Proposition \ref{smallPPCmatleast2}, is at most $4\sqrt{\epsilon} N$. The first inequality of the lemma follows. 

\vs

For the second inequality of the lemma, we argue as above, except this time using Proposition \ref{smallbad2interval}: \begin{align*} &\sum_{I \in \mc{I}^N} \sum_{k \text{ sandwiched}} \PPC^{J_{k-1}^I,J_{k+1}^I}(0,a)\\ &\qquad \qquad \qquad \le \sum_{I \in \mc{I}^N} \sum_{k \text{ sandwiched}} |J_{g_2(k)}^I|\hs |J_{b_2(k)}^I| \\ &\qquad \qquad \qquad \le \left(\sum_{I \in \mc{I}^N} \sum_{k \text{ sandwiched}} |J_{g_2(k)}^I|^2\right)^{1/2}\left(\sum_{I \in \mc{I}^N} \sum_{k \text{ sandwiched}} |J_{b_2(k)}^I|^2\right)^{1/2} \\ &\qquad \qquad \qquad \le \left(2\sum_{I \in \mc{I}^N} \sum_k |J_k^I|^2\right)^{1/2}\left(\sum_{I \in \mc{I}^N} \sum_k 2 \PPC^{J_{k-1}^I,J_{k+1}^I}(\frac{1}{2},\frac{3}{2})\right)^{1/2},\end{align*} where the last inequality used $\summ(J_{k-1}\cup J_k\cup J_{k+1}) \le \frac{3}{2}$. Now just observe $$\sum_{I \in \mc{I}^N} \sum_k 2 \PPC^{J_k^I,J_{k+1}^I}(\frac{1}{2},\frac{3}{2}) \le 2\sum_{n \le N} \sum_{2 \le m \le N-n+1} 1_{g_n+\dots+g_{n+m-1} \in (\frac{1}{2},\frac{3}{2})},$$ which, by Proposition \ref{smallPPCmatleast2}, is at most $4\sqrt{\epsilon} N$. The second inequality of the lemma follows.
\end{proof}

\vspace{1.5mm}

We are ready to complete the proof of Proposition \ref{partition}.

\begin{proposition}\label{sizeofintervalsbounds} 
For all $N$ large, $$\sum_{I \in \mc{I}^N}\sum_k \binom{|J_k^I|+1}{2} \geq  (\frac{1}{2}-4\sqrt{2}\epsilon^{1/4})N.$$
\end{proposition}

\begin{proof}
Using Lemma \ref{trichotomy} we have that 
$$(\frac{1}{2}+o(1))N = \sum_{I\in \mc I^N}\PPC^I(0,\frac{1}{2}) = \sum_I\sum_{k=1}^{r_I} \PPC^{J^I_k}(0,\frac{1}{2}) +\sum_{k=1}^{r_I-1}\PPC^{J^I_k,J^I_{k+1}}(0, \frac{1}{2})$$ $$\qquad \qquad \qquad \qquad \qquad \qquad \qquad \qquad \qquad \qquad \qquad + \sum_{k \text{ sandwiched}}\PPC^{J_{k-1},J_{k+1}}(0,\frac{1}{2}).$$ 

Invoking Proposition \ref{PPCcrosstermbounds} in the first line and Lemma \ref{Ibounds} in the second we have: \begin{align*}(\frac{1}{2}+o(1))N-\sum_{I \in \mc{I}^N}\sum_k \binom{|J_k^I|+1}{2} &\le 4\sqrt 2 \epsilon^{1/4}\left(\sum_{I \in \mc{I}^N} \sum_k |J_k^I|^2\right)^{1/2}\sqrt{N} \\
& \le  8\epsilon^{1/4}\left(\sum_{I \in \mc{I}^N} \sum_k \binom{|J_k^I|+1}{2}\right)^{1/2}\sqrt{N}.
\end{align*}

Writing $$\sum_{I \in \mc{I}^N} \sum_k \binom{|J_k^I|+1}{2} = (\frac{1}{2}-\delta) N,$$ we see $$\delta+o(1)\le 8 \epsilon^{1/4}(\frac{1}{2} - \delta)^{1/2},$$ which yields (after a little computation) $\delta \le 4\sqrt{2}\epsilon^{1/4}$ for $N$ large enough.
\end{proof}

\vspace{1.5mm}

\vsss

\section{Appendix: Proof of Lemma \ref{5/12lemma}}\label{appendix}

We restate Lemma \ref{5/12lemma} for the reader's convenience.

\vspace{1mm}

\begin{lemma-repeat}
For positive integers $a,b,c,L$ satisfying $1 \le a \le b \le c \le L$, we have  $$(a-1)a+(b-a)(b-a+1)+(c-b)(c-b+1)+(L-c)(L-c+1) \qquad \qquad$$ $$\qquad \qquad \qquad \qquad +(a-1)(b-a)+(b-a)(c-b)+(c-b)(L-c) \ge \frac{5}{12}L^2+\frac{1}{6}L-\frac{7}{12}.$$
\end{lemma-repeat}

\begin{proof}
Fix $L \ge 1$. It clearly suffices to prove the inequality for all real numbers $a,b,c$ satisfying $1 \le a \le b \le c \le L$. By compactness, we may work with a triple $(a,b,c)$ that achieves the minimum value of the left hand side (which we denote $\LHS$) minus the right hand side (which we denote $\RHS$). We divide into three cases. 

\vs

\tb{Case 1}: $c=L$. 

\vspace{1.5mm}

As one may compute, $\frac{\partial}{\partial c}[\LHS-\RHS] = -a+2c-L = L-a$. 

\vspace{1.5mm}

\textit{Subcase 1}: $a=L$. Then $b=L$, which gives $\LHS-\RHS = \frac{28L^2-51L+28}{48}$, which is non-negative, since it is equal to $5$ at $L=1$ and has derivative $56L-51$, which is positive for $L \ge 1$.

\vspace{1.5mm}

\textit{Subcase 2}: $a \not = L$. Then $\frac{\partial}{\partial c}[\LHS-\RHS] > 0$, so since $(a,b,c)$ is a minimizer, we must have $b=c$, for otherwise we can decrease $c$ a bit to decrease $\LHS-\RHS$. Thus, $\LHS-\RHS = a^2-(L+1)a+\frac{28L^2-3L+28}{48}$, which has minimum occurring at $a = \frac{L+1}{2}$, which gives $\LHS-\RHS = \frac{1}{3}L^2-\frac{9}{16}L+\frac{1}{3}$, which is always non-negative, since it is at $L=1$ and the derivative is $\frac{2}{3}L-\frac{9}{16}$, which is non-negative for $L \ge 1$.

\vsss

\tb{Case 2}: $c \not = L$ and $b = c$.

\vspace{1.5mm}

Since $c \not = L$ and $(a,b,c)$ is a minimizer, we must have $\frac{\partial}{\partial c}[\LHS-\RHS] \ge 0$, for otherwise we could increase $b,c$ a bit to decrease $\LHS-\RHS$. Recall $\frac{\partial}{\partial c}[\LHS-\RHS] = -a+2c-L$; so, $2c-L \ge a$.

\vspace{1.5mm}

\textit{Subcase 1}: $a=b$. In this case, $\LHS-\RHS = 2c^2-(2L+2)c+\frac{7L^2+10L+7}{12}$, which has minimum at $c = \frac{L+1}{2}$, which yields $\LHS-\RHS = \frac{(L-1)^2}{12}$, which is non-negative.

\vspace{1.5mm}

\textit{Subcase 2}: $a \not = b$. In this case, we must have $\frac{\partial}{\partial b}[\LHS-\RHS] \le 0$, since otherwise we could decrease $b$ a bit to decrease $\LHS-\RHS$. Note $\frac{\partial}{\partial b}[\LHS-\RHS] = -1+2b-L$, so since $2b-L \ge a$, we must have $a=1$ and $2b-L = 1$. So, we have $a=1,b=\frac{L+1}{2},c=\frac{L+1}{2}$, which indeed has $\LHS-\RHS \ge 0$.

\vsss

\tb{Case 3}: $c \not = L$ and $b \not = c$. 

\vspace{1.5mm}

In this case, $\frac{\partial}{\partial c}[\LHS-\RHS]$ must be $0$, for otherwise we could perturb $c$ a bit to decrease $\LHS-\RHS$. So, $-a+2c-L = 0$. 

\vspace{1.5mm}

\textit{Subcase 1}: $a = 1$. Having $a=1$ and $-a+2c-L=0$, i.e., $c = \frac{L+1}{2}$, yields $\LHS-\RHS = \frac{1}{4}[4b^2-4(L+1)b+3L^2+2L-1]$, which is minimized as $b = \frac{L+1}{2}$, which was dealt with in Subcase 2 of Case 2.

\vspace{1.5mm}

\textit{Subcase 2}: $a = b$. Then, $48(\LHS-\RHS) = 84b^2-(72L+96)b+16L^2+45L+28$, which has minimum at $b = \frac{72L+96}{84}$, which gives $48(\LHS-\RHS) = 16L^2+45L+28$, which is clearly non-negative.

\vspace{1.5mm}

\textit{Subcase 3}: $a \not \in \{1,b\}$. Then $\frac{\partial}{\partial a}[\LHS-\RHS] = 0$, for otherwise we could perturb $a$ a bit to decrease $\LHS-\RHS$. So, $-1+2a-c = 0$. Together with $-a+2c-L = 0$ yields $a = \frac{L+2}{3}, c = \frac{2L+1}{3}$, which yields $\LHS-\RHS = b^2-(L+1)b+\frac{12L^2+29L+12}{48}$, which is minimized at $b = \frac{L+1}{2}$, which yields $\frac{5L}{48}$, which is clearly non-negative. 
\end{proof}

\vs

\section{Acknowledgments}\label{acknowledgments}

We would like to thank our advisor Ben Green for suggesting this problem to us.

\vs


\begin{thebibliography}{10}

\bibitem{abr} C. Aistleitner, V. Blomer, M. Radziwi\l\l. Triple correlation and long gaps in the spectrum of flat tori. ArXiv e-prints, September 2018, 1809.07881.

\bibitem{all} C. Aistleitner, G. Larcher, M. Lewko. Additive Energy and the Hausdorff dimension of the exceptional set in metric pair correlation problems, with an Appendix by Jean Bourgain. Israel J. Math. 222 (1) (2017) 463-485. 

\bibitem{alp} C. Aistleitner, T. Lachmann, F. Pausinger. Pair correlations and equidistribution, J. Number Theory 182, (2018) 206–220.

\bibitem{emm} A. Eskin, G. Margulis, S. Mozes. Quadratic forms of signature (2, 2) and eigenvalue spacings on rectangular 2-tori. Ann. of Math. (2) 161 (2005), 679-725.

\bibitem{green} B. Green, personal communication.

\bibitem{heath-brown} D. R. Heath-Brown. Pair correlation for fractional parts of $\alpha n^2$, Math. Proc. Cambridge Philos. Soc. 148 (2010), 385–407.

\bibitem{udbook} L. Kuipers, H. Niederreiter. Uniform distribution of sequences, Mineola, NY, Dover Publications (2006).‌

\bibitem{rudnicksarnak} Z. Rudnick, P. Sarnak. The pair correlation function of fractional parts of polynomials, Comm. Math. Phys. 194 (1998), 61–70.

\bibitem{rz1999} Z. Rudnick, A. Zaharescu. A metric result on the pair correlation of fractional parts of sequences, Acta. Arith. 89 (1999), 283–293.

\bibitem{rz2002} Z. Rudnick, A. Zaharescu. The distribution of spacings between fractional parts of lacunary sequences, Forum Math. 14 (2002), 691–712.

\bibitem{sarnak} P. Sarnak, Values at integers of binary quadratic forms. Harmonic analysis and number
theory (Montreal 1996), 181-203, CMS Conf. Proc. 21, Amer. Math. Soc., Providence, RI, 1997.

\bibitem{steinerberger} S. Steinerberger, Poissonian pair correlation and discrepancy, Indag. Math. 29 (2018) 1167-1178.

\bibitem{walker} A. Walker. The Primes are not metric Poissonian, Mathematika 64 (2018) 230-236.

\end{thebibliography}
\end{document}